\documentclass[12pt,reqno]{amsart}
\usepackage{amsmath,amsfonts,amssymb,amsrefs,fullpage,esint} 

\usepackage{amsthm}
\usepackage{tabularx}
\usepackage{tikz}
\usepackage{array}
\usepackage{caption} 
\usepackage{tablefootnote}
\usepackage{multirow}
\usepackage{enumerate}
\usepackage{float}
\AtBeginDocument{%
   \def\MR#1{}
}

\newtheorem{thm}{Theorem}
\newtheorem{prop}[thm]{Proposition}
\newtheorem{cor}[thm]{Corollary}
\newtheorem{lem}[thm]{Lemma}
\newtheorem*{acknowledgement}{Acknowledgements}

\DeclareFontFamily{U}{mathx}{\hyphenchar\font45}
\DeclareFontShape{U}{mathx}{m}{n}{
      <5> <6> <7> <8> <9> <10>
      <10.95> <12> <14.4> <17.28> <20.74> <24.88>
      mathx10
      }{}
\DeclareSymbolFont{mathx}{U}{mathx}{m}{n}
\DeclareFontSubstitution{U}{mathx}{m}{n}
\DeclareMathAccent{\widecheck}{0}{mathx}{"71}
\DeclareMathAccent{\wideparen}{0}{mathx}{"75}

\numberwithin{equation}{section}
\numberwithin{thm}{section}

\newcommand{\R}{\mathbb{R}}
\newcommand{\T}{\mathbb{T}}

\newcommand{\Z}{\mathbb{Z}}

\newcommand{\e}{e^{it\Delta}}

\newcommand{\tn}{\theta,\nu}
\newcommand{\jtg}{{j,\text{tang}}}
\newcommand{\jtr}{{j,\text{trans}}}
\DeclareMathOperator{\rd}{RapDec}

\newcommand{\eps}{\varepsilon}

\title{$L^p$ Estimates of the Maximal Schr\"odinger operator in $\R^n$}
\author{Xiumin Du}
\author{Jianhui Li}
\begin{document}

\begin{abstract}
We obtain $L^p$ estimates of the maximal Schr\"odinger operator in $\R^n$ using polynomial partitioning, bilinear refined Strichartz estimates, and weighted restriction estimates. 
\end{abstract}

\maketitle

\section{Introduction}
Consider the solution to the free Schr\"odinger equation in $\R^{n+1}$:
$$
\e f(x) = (2\pi)^{-n}\int e^{i(x\cdot \xi + t|\xi|^2)} \hat{f}(\xi) d\xi.
$$
Carleson proposed in \cite{Ca80} the problem of finding the smallest $s$ such that the solution $\e f(x)$ converges to $f(x)$ pointwise as $t \to 0$ whenever $f\in H^s(\R^n)$. Important milestones of the problem have been established in \cites{Ca85,Co83,Sj87,Ve88,Bo95,MVV96,TV2000,Lee06,Bo12,LR17,LR19.2,Bo16,DGL17,DGLZ18,DZ19}. The problem has been solved except for the endpoint. We summarize the best-known results in the table below.

\begin{table}[h!]
\begin{tabular}{|l|l|l|}
\hline
$n$ & Necessity of $s \geq \frac{n}{2(n+1)}$ & Sufficiency of $s > \frac{n}{2(n+1)}$ \\ \hline
$1$      & Dahlberg and Kenig \cite{DK82} & Carleson\tablefootnote{Carleson also proved the sufficiency of the endpoint case $s = \frac{1}{4}$ when $n=1$.} \cite{Ca80}         \\ \hline
$2$      & \multirow{2}{*}{Bourgain \cite{Bo12}}           & Du, Guth and Li \cite{DGL17} \\ \cline{1-1} \cline{3-3} 
$\geq 3$ &           & Du and Zhang \cite{DZ19}     \\ \hline
\end{tabular}
\caption{The best-known result to the Carleson's problem.}
\label{}
\end{table}

As a common strategy to approach pointwise convergence, Stein's maximal principle reduces the qualitative problem to proving a quantitative $L^p$ estimate on the maximal operator for some $p\geq 1$:
\begin{equation}
    \left \|\sup_{0<t\leq 1} |e^{it\Delta}f| \right\|_{L^p(B^{n}(0,1))} \lesssim \|f\|_{H^s(\R^n)}.
\end{equation}
Here, $B^n(0,R)$ denotes the ball of radius $R$ in $\R^n$ centered at the origin. By Littlewood--Paley decomposition, parabolic rescaling, and a localization argument \cite{Lee06}, in the case $p\geq 2$ it can be further reduced to estimates of the form:

\begin{equation}\label{eq:intro_Lp_maximal_s}
    \left\|\sup_{0<t\leq R}|\e f|\right\|_{L^p(B^{n}(0,R))} \lesssim R^{s - n\left  ( \frac{1}{2} -\frac{1}{p}\right )}\|f\|_{2}, \quad \forall R \geq 1,
\end{equation}
and for all $f \in L^2$ Fourier supported on $B^{n}(0,1)$. 

Estimates of the form \eqref{eq:intro_Lp_maximal_s} with $s \geq \frac{n}{2(n+1)}$ and $s - n  ( \frac{1}{2} -\frac{1}{p} ) \geq 0 $ were believed to be true by testing with standard examples used in restriction theory. This is indeed the case when $n=1$ \cite{KPV91}, and $n=2$ \cite{DGL17} up to a factor of $R^{\varepsilon}$ loss. More precisely, \eqref{eq:intro_Lp_maximal_s} holds with $(n,s,p)= (1,1/4,4), (2,1/3^+,3)$, which implies sharp estimates for all $p$. However, in higher dimensions only the $L^2$ sharp estimate has been established \cite{DZ19}: up to a factor of $R^{\varepsilon}$ loss, \eqref{eq:intro_Lp_maximal_s} holds with $p=2$ and $s=\frac{n}{2(n+1)}$. Moreover, Kim, Wang, Zhang, and the first author proved in \cite{DKWZ20} by looking at Bourgain's counterexample \cite{Bo12} in every intermediate dimension that 

\begin{equation}\label{eq:intro_Lp_maximal}
    \left\|\sup_{0<t\leq R}|\e f|\right\|_{L^p(B^{n}(0,R))} \lesssim_\varepsilon R^{\varepsilon} \|f\|_{2}
\end{equation}
fails if
\begin{equation}
    p < p_{ness} : = \max_{1\leq m \leq n \text{ integers}} 2+ \frac{4}{n-1+m+\frac{n}{m}}.
\end{equation}
Since then, efforts \cites{Wu21,CMW23} have been made to find the smallest exponent $p$ for which \eqref{eq:intro_Lp_maximal} holds. The previous best known result \cite{CMW23} was
\begin{equation}
    p \geq p_{cmw} : = 2+ \frac{4}{n+1+\frac{1}{2} + ... +\frac{1}{n}}.
\end{equation}
Table \ref{tab:comparison} below shows a list of values $p_{cmw}$ and $p_{ness}$ for $3\leq n \leq 10$.

To introduce the terminology for the range of the Lebesgue exponents $p$ we obtain, we define the following important exponents, which will also be used in the proof of the main theorem below. 

Let $m$ be an intermediate dimension satisfying $3\leq m \leq n+1$. Let
\begin{equation}\label{eq:pwm}
    p_{w,m}:= \left (\frac{1}{2} -\frac{1}{m(m-1)(m-2)}\right )^{-1}
\end{equation}
and
\begin{equation}
     \gamma_{w,m} := \frac{m-1}{2m} - \frac{n+m-1}{2m(m-1)(m-2)}>0.
\end{equation}
These exponents $p_{w,m}$ and $\gamma_{w,m}$ come from the weighted $L^p$ estimate developed by Wang, Zhang and the authors in \cite{DLWZ}. See Theorem \ref{thm:LWRE} below for the precise statement.

Let $M \leq n$ to be determined. 
Let
\begin{equation}\label{eq:def_p_m,alpha_m}
    p_m := \frac{2m}{m-1}, \quad \alpha_m := \frac{m^2-3m}{m^2-2m-2}, \quad 3\leq m \leq M.
\end{equation}
Here $\alpha_m$ is chosen such that $ \frac{\alpha_m}{p_{m-1}} +  \frac{1-\alpha_m}{p_{w,m}} = \frac{1}{p_m}$.

Next, let 
\begin{equation} \label{eq:gamma_m}
    \gamma_m := 
\begin{cases}
    \frac{2-n}{12} \quad &\text{ if } m=3; \\
    \alpha_m \gamma_{m-1} + (1-\alpha_m) \gamma_{w,m} \quad &\text{ if } 4\leq m \leq M; \\
    0 \quad &\text{ if } M+1 \leq m \leq n+1,
\end{cases}
\end{equation}
and choose $M$ to be the largest integer so that $\gamma_m$ is negative for all integers $m \in [3,M]$.

Finally, choose
\begin{equation}\label{eq:alpha_M+1}
    \alpha_{M+1}  := \frac{\gamma_{w,M+1}}{\gamma_{w,M+1}-\gamma_M} \in (0,1)
\end{equation}
so that $0=\gamma_{M+1} = \alpha_{M+1}\gamma_{M} + (1-\alpha_{M+1}) \gamma_{w,M+1}$, and define
\begin{equation} \label{eq:p_n+1}
    p_{n+1} = p_{n} = ... = p_{M+1} =  \left (\frac{\alpha_{M+1}}{p_{M}} +  \frac{1-\alpha_{M+1}}{p_{w,M+1}}\right)^{-1}.
\end{equation}

\begin{thm}\label{thm:main}
    Let $n\geq 3$, and $p_{n+1}$ be given as above. For $p \geq p_{n+1}$, the $L^p$ estimate on the maximal Schr\"odinger operator
    \begin{equation}\label{eq:intro_Lp_maximal'}
    \left\|\sup_{0<t\leq R}|\e f|\right\|_{L^p(B^{n}(0,R))} \lessapprox \|f\|_{2}
\end{equation}
holds for any $f\in L^2$ Fourier supported on $B^n(0,1)$.  
\end{thm}
Here we write $A\lessapprox B$ if $A\leq C_\epsilon R^\epsilon B$ for any $\epsilon>0$ and any $R\geq 1$. Table \ref{tab:comparison} below shows a list of values $p_{n+1}$ for $3\leq n \leq 10$.

\def\arraystretch{1.7}
\begin{table}[h!]
\begin{tabular}{|m {1.5cm} |m {4cm}|m {4cm}|m {4cm}|}
\hline
$n$ & $p_{n+1}$                             & $p_{cmw}$  & $p_{ness}$ \\ \hline
3   & $2+\frac{26}{35} \approx 2.7429$      & $2+\frac{24}{29} \approx2.8276$   & $2+\frac{8}{11} \approx2.7273$ \\ \hline
4   & $2+\frac{146}{245} \approx 2.5959$    & $2+\frac{48}{73} \approx2.6575$   & $2+\frac{4}{7} \approx2.5714$  \\ \hline
5   & $2+\frac{730}{1459} \approx 2.5003$   & $2+\frac{240}{437} \approx2.5492$ & $2+\frac{8}{17} \approx2.4706$ \\ \hline
6   & $2+\frac{42118}{97051} \approx 2.434$ & $2+\frac{80}{169} \approx2.4734$  & $2+\frac{2}{5} =2.4$     \\ \hline
7  & $2+\frac{239324}{623753} \approx 2.3837$      & $2+\frac{560}{1343} \approx2.417$     & $2+\frac{6}{17} \approx2.3529$ \\ \hline
8  & $2+\frac{119662}{347269} \approx 2.3446$      & $2+\frac{1120}{3001} \approx2.3732$   & $2+\frac{6}{19} \approx2.3158$ \\ \hline
9  & $2+\frac{39451162}{125976997} \approx 2.3132$ & $2+\frac{10080}{29809} \approx2.3382$ & $2+\frac{2}{7} \approx2.2857$  \\ \hline
10 & $2+\frac{19725581}{68686691} \approx 2.2872$  & $2+\frac{10080}{32581} \approx2.3094$ & $2+\frac{6}{23} \approx2.2609$ \\ \hline
\end{tabular}
\caption{Comparison on the Lebesgue exponents among our result $p_{n+1}$ , the previous best known result $p_{cmw}$ and the best known necessary condition $p_{ness}$ for $3\leq n \leq 10$.}
\label{tab:comparison}
\end{table}

The asymtoptics of $p_{n+1}$ and $p_{cmw}$ are both
$$
p_{n+1} , p_{cmw} = 2+\frac{4}{n} + \frac{4 \log n}{n^2} + O(n^{-2}),
$$
while that of $p_{ness}$ is
$$
p_{ness} = 2 + \frac{4}{n}- \frac{8}{n^{3/2}} + O(n^{-2}).
$$
Although $p_{n+1} < p_{cmw}$ for all $n \geq 3$, the improvement is insignificant when $n$ is large. This is because both methods rely on bilinear estimates. See also the last paragraph of Section \ref{sec:Sketch_of_pf} for more discussion.

\subsection{A general version}

Consider the Fourier extension operator $E$ on the truncated paraboloid defined on $L^2$ functions $f$:
\begin{equation}
    Ef(x',x_{n+1}) : = \int_{B^n(0,1)} e^{i (x' \cdot \xi + x_{n+1}|\xi|^2)} f(\xi) d \xi, \quad (x',x_{n+1}) \in \R^{n+1}.
\end{equation}

The techniques in this paper naturally extend to the Fourier extension operator on hypersurfaces in $\R^{n+1}$ with positive definite second fundamental form. For simplicity and the application to the Schr\"odinger estimate, we will focus on the paraboloid.

We say that $\mu$ is an $n$-dimensional measure in $\R^{n+1}$ if it is a probability measure supported on $B^{n+1}(0,1)$ and satisfies that
$$
\mu(B^{n+1}(x,r)) \leq C_{\mu}r^{n}, \quad \forall r>0, \quad \forall x \in \R^{n+1}.
$$
Denote 
\begin{equation}\label{eq:mu_R}
    d\mu_R(\cdot): = R^{n}d\mu(\frac{\cdot}{R}).
\end{equation}

\begin{thm}\label{thm:main_ext}
    Let $n \geq 3$, and $p_{n+1}$ be given as in equations \eqref{eq:pwm} to \eqref{eq:p_n+1}. Let $\mu$ be an $n$-dimensional measure in $\R^{n+1}$ and let $\mu_R$ be defined as in $\eqref{eq:mu_R}$. For $p \geq p_{n+1}$, the $L^p$ weighted restriction estimate
    \begin{equation}\label{eq:thm:main_ext}
        \|Ef\|_{L^p(d \mu_R)} \lessapprox \|f\|_2
    \end{equation}
    holds for any $f\in L^2(B^n(0,1))$.
\end{thm}

Note that $E(\hat{f})(x',t) = (2\pi)^n \e f(x')$ if $f$ has Fourier support on the unit ball. By the locally constant property of $E$, the estimate \eqref{eq:intro_Lp_maximal'} is a special case of the weighted restriction estimate \eqref{eq:thm:main_ext} (see \cite{DZ19} for instance). Our goal is to prove Theorem \ref{thm:main_ext}.

\subsection{Sketch of the proof of Theorem \ref{thm:main_ext}.}\label{sec:Sketch_of_pf}
The proof of Theorem \ref{thm:main_ext} relies on the method of polynomial partitioning developed by Guth \cites{Gu16, Gu18}, the bilinear refined Strichartz estimate obtained by Guth, Li, and the first author \cite{DGL17}, and the weighted $L^p$ estimates from Wang, Zhang, and the authors \cite{DLWZ}.

Polynomial partitioning has been extensively used to obtain restriction estimates for the paraboloid \cites{Gu16,Gu18,HR19,DGOWWZ21, WangRestr, WangWuRestr, GWZRestr}, the cone \cite{OW22}, and H\"omander-type oscillatory integral operators \cite{GHI19}. The almost sharp Schr\"odinger maximal estimate for $n=2$ in \cite{DGL17} and the partial progress \cite{CMW23} in higher dimensions also rely on the method of polynomial partitioning.

Before setting up the induction, we will introduce the weighted broad norm $BL^p(d\mu_R)$ following \cite{DGOWWZ21}. The quantity $\|Ef\|_{BL^p(d\mu_R)}$ measures the contribution to $\|Ef\|_{L^p(d\mu_R)}$ from the \textit{broad} part, which is large when there are portions of $f$ with separated Fourier supports that contribute significantly. The precise definition of $BL^p(d\mu_R)$ is given in Subsection \ref{sec:weightedBL}. The leftover \textit{narrow} part can be handled relatively easily by parabolic rescaling and induction on scales.

To estimate the broad part, we shall induct on the dimension of the algebraic variety near which $Ef$ is concentrated. We will prove that for each intermediate dimension $m \in [3, M]$ and $p_m = \frac{2m}{m-1}$, we have
\begin{equation}\label{eq:intro_induct}
\left\|Ef\right\|_{BL^{p_m}(d\mu_R)} \lessapprox R^{\gamma_m} \|f\|_2
\end{equation}
whenever most of the wave packets of $f$ are tangent to an $m$-dimensional algebraic variety $Z_m$. Here, the exponent $\gamma_m$ defined as in \eqref{eq:gamma_m} is non-positive. Readers should refer to Subsections \ref{sec:WPDec} and \ref{sec:tang_wavepackets} for the detailed discussion on wave packet decomposition and the precise definition of a wavepacket being tangent to an algebraic variety.

The base case of the induction, where $m=3$, follows from polynomial partitioning and a weaker version of the bilinear refined Strichartz estimate developed in \cite{DGL17}. See Theorem \ref{thm:BWRE} and Corollary \ref{cor:BWRE} below. By induction, we may assume that the $(m-1)$-dimensional estimate is true. Polynomial partitioning comes in to help reduce the dimension of the algebraic variety near which the wave packets of $f$ are concentrated.

Following Guth \cite{Gu18}, $\mathbb{R}^{n+1}$ can be divided into $\sim D^{m}$ pieces by the zero set $Z(P)$ of a polynomial $P$ of degree at most $D$, such that the weighted broad norm of $Ef$ is simultaneously divided into roughly equal pieces. One of the following three scenarios occurs:
\begin{itemize}
    \item \textit{Cellular/non-algebraic case}: the wave packets of $f$ are concentrated away from a neighborhood of the zero set $Z(P)$. The argument here is the same as \cite{Gu18}*{Subsection 8.1}. The induction closes as long as 
    \begin{equation}\label{ineq:cell_rest}
        p_m \geq \frac{2m}{m-1}.
    \end{equation}
    \item \textit{Algebraic case}: the wave packets of $f$ are concentrated in a neighborhood of an algebraic variety $Y$ of dimension $m-1$.
    \begin{itemize}
        \item \textit{Transverse sub-case}: most of the wave packets of $f$ pass through $Y$ transversely. The argument here is the same as \cite{Gu18}*{Subsection 8.4}. The induction closes as long as
        \begin{equation}\label{ineq:transverse_rest}
            \frac12(n+1-m)\left (\frac{p_m}{2}-1\right)+p_m\gamma_m \geq 0.
        \end{equation}
        Direct computation shows that inequality \eqref{ineq:transverse_rest} is satisfied. In fact, we have the freedom to choose the pair of exponents satisfying \eqref{ineq:cell_rest} and \eqref{ineq:transverse_rest}. It turns out that choosing the exponent $p_m = \frac{2m}{m-1}$ minimizes the exponent $p$ in the $L^p$ estimate \ref{eq:intro_Lp_maximal'} we are interested in. See Subsection \ref{sec:optimal} for details.
        \item \textit{Tangential sub-case}: most of the wave packets of $f$ are tangent to $Y$. In this case, we shall apply the induction hypothesis. Note that $p_{m-1} = \frac{2(m-1)}{m-2} > \frac{2m}{m-1} = p_{m}$. We obtain the desired estimate by interpolating the induction hypothesis with the following weighted $L^p$ estimate
        $$
        \|Ef\|_{L^{p_{w,m}}(B^{n+1}(0,R),d\mu)} \lessapprox R^{\gamma_{w,m}}\|f\|_{2},
        $$
        whenever most of the wave packets of $f$ are tangent to an $(m-1)$-dimensional algebraic variety. See Theorem \ref{thm:LWRE}.
    \end{itemize}
\end{itemize}

We have obtained \eqref{eq:intro_induct} for $m \in [3,M]$. The reason for not continuing the same process is that the exponent $\gamma_m$ resulting from the interpolation will always be positive if $\gamma_{m-1}$ is non-negative. There is no way to recover the loss of a positive power of $R$. Therefore, we terminate the process at $m=M$, where $\gamma_M$ is the last negative exponent. We then follow the polynomial partitioning, with $\gamma_{M+1} = \gamma_{M+2} = ... = \gamma_{n+1} = 0$ and $p_{M+1} = p_{M+2} = ... =p_{n+1}$, where $p_{n+1}$ is defined as in \eqref{eq:p_n+1}. The definition of $p_{n+1}$ comes from the interpolation in the tangential sub-case of the polynomial partitioning at $m=M+1$. Therefore, we conclude Theorem \ref{thm:main_ext}.

We remark that the strategy described above can be adapted to prove maximal Schr\"odinger estimates of the form
\begin{equation}
    \left\|\sup_{0<t\leq R}|Ef|\right\|_{L^p(B^{n}(0,R))} \lessapprox  R^{\gamma_0} \|f\|_{2}
\end{equation}
for some fixed $\gamma_0>0$. The smallest $p\geq\frac{2(n+1)}{n}$ we can obtain using the same ingredients can be computed by modifying the definition of $M$. In this general case, we want $\gamma_m < \gamma_0$ for all integers $m \in [3,M]$.

A slight modification of the strategy above also yields weighted $k$-broad estimates, following Guth \cite{Gu18}. However, passing from weighted $k$-broad estimates to weighted linear estimates remains a challenging problem, with no prior work addressing this issue. The authors will pursue this direction in a subsequent paper.

\subsection{Organization of the paper}In Section \ref{sec:prelim}, we will discuss several preliminaries, including wave packet decomposition, the definition of wave packets being tangential to an algebraic variety, the broad norm $BL^p$ and its relation to $L^p$ estimates. In Section \ref{sec:WRE}, we will introduce the two main ingredients for our proof: linear and bilinear weighted $L^p$ restriction estimates. In Section \ref{sec:PolyPart}, using polynomial partitioning we prove a main inductive proposition that implies Theorem \ref{thm:main_ext}.

\subsection{List of notation.}  
Throughout the article, we write $A\lesssim B$ if $A\leq CB$ for some absolute constant $C$; $A\sim B$ if $A\lesssim B$ and $B\lesssim A$; $A\lesssim_\eps B$ if $A\leq C_\eps B$; $A\lessapprox B$ if $A\leq C_\eps R^\eps B$ for any $\eps>0, R>1$.

For a large parameter $R$, ${\rm RapDec}(R)$ denotes those quantities that are bounded by a huge (absolute) negative power of $R$, i.e. ${\rm RapDec}(R) \leq C_N R^{-N}$ for arbitrarily large $N>0$. Such quantities are negligible in our argument.

For each $\varepsilon>0$, we will choose a sequence of small parameters
$$
\delta_{\text{deg}} \ll \delta \ll \delta_{n} \ll .. \ll \delta_1\ll \delta_0 \ll \varepsilon.
$$

For polynomials $P_1,...,P_{n+1-m}$, the algebraic variety $Z=Z(P_1,...,P_{n+1-m})$ is defined to be the set of common zeros of $P_1,...,P_{n+1-m}$. Let $D_Z$ denote an upper bound of the degrees of $P_1,...,P_{n+1-m}$. Here $D_Z \leq R^{\delta_{\text{deg}}}$ unless otherwise specified.

\begin{acknowledgement}
XD is supported by NSF CAREER DMS-2237349 and Sloan Research Fellowship. JL is supported by AMS-Simons Travel Grant.
\end{acknowledgement}

\section{Preliminaries}\label{sec:prelim}

\subsection{Wave packet decomposition}\label{sec:WPDec}
The setup is similar to \cite{Gu18}. For any $L^2$ function $f$ on $B^n(0,1)$, we decompose $f$ into pieces $f_{\tn}$, each of which is essentially localized in both physical space and frequency space. More precisely, cover $B^n(0,1)$ by finitely overlapping balls $\theta$ of radius $R^{-1/2}$ and cover $\mathbb{R}^{n}$ by finitely overlapping balls of radius $R^{1/2 +\delta}$, centered at $\nu \in R^{1/2 +\delta}\mathbb Z^{n}$, where $\delta$ is a small positive quantity that may depend on $\varepsilon$. By partition of unity, we have a decomposition
$$
f = \sum_{(\tn) \in \mathbb T} f_{\tn} + \text{RapDec}(R)\|f\|_{2},
$$
where $f_{\tn}$ is supported in $\theta$, and up to a rapidly decaying tail, has Fourier transform supported in a ball of radius $R^{1/2 +\delta}$ around $\nu$.

A key feature of the decomposition is the almost $L^2$ orthogonality among $f_{\tn}$: for each subset $\T ' \subset \T$, we have
$$
\left\|\sum_{(\tn) \in \T' } f_{\tn}\right\|_{L^2}^2 \sim \sum_{(\tn)\in \T'} \|f_{\tn} \|_{2}^2.
$$
And for each pair $(\tn)$, $Ef_{\tn}$ restricted to $B(0,R)$ is roughly supported on $T_{\tn}$, where
$$
T_{\theta ,\nu} := \{ (x',x_{n+1}) \in B^{n+1}(0,R) : |x' + 2 x_{n+1} c(\theta) - \nu| \leq R^{1/2 +\delta}\}
$$
is a tube of radius $R^{1/2 +\delta}$ centered at $\nu$, pointing to the direction 
$$G(\theta):= \frac{(-2c(\theta),1)}{|(-2(\theta),1)|},$$
of length $\sim R$. Here $c(\theta)$ is the center of $\theta$.

\subsection{Wave packets tangential to an algebraic variety}\label{sec:tang_wavepackets}
In the introduction, we discussed the concept of wave packets being \textit{tangent} to an algebraic variety, which we will make precise here. A variety $Z(P_1,...,P_{n+1-m})$ is called a \textit{transverse complete intersection} if 
$$
\nabla P_1 \wedge ... \wedge \nabla P_{n+1-m} \neq 0 \quad \text{on } Z(P_1,...,P_{n+1-m}).
$$

Let $Z$ be an algebraic variety and $N$ a positive parameter. For any $(\tn) \in \T$, the wavepacket $T_{\tn}$ is said to be $NR^{-1/2}$\textit{-tangent} to $Z$ if 

$$
T_{\tn} \subset N_{NR^{1/2}}Z, \quad \text{and} \quad \text{Angle}(G(\theta),T_{z}Z) \leq NR^{-1/2}
$$
for any non-singular point $z \in N_{2NR^{1/2}}(T_{\tn}) \cap (B^{n+1}(0,2R))\cap Z$. Let 
$$
\T_Z(N): = \{(\tn)\in \T \mid T_{\tn} \text{ is } NR^{-1/2}\text{-tangent to }Z\}.
$$
We say that $f$ is \textit{concentrated in wave packets} from $\T_Z(N)$ if 
$$
\sum_{(\tn) \not \in \T_Z(N)} \|f_{\tn}\|_{L^2} \leq \text{RapDec}(R) \|f\|_{2}. 
$$
It is worth noting that the smallest interesting value of $N$ is $R^{\delta}$ as the radius of $T_{\tn}$ is $R^{1/2+\delta}$.

\subsection{Weighted broad norm} \label{sec:weightedBL} Following \cite{DGOWWZ21}, we define the weighted broad norm in this subsection.

By locally constant property, see for instance \cite{DZ19}, the $L^p$ estimate on the maximal operator \eqref{eq:intro_Lp_maximal'} is a corollary of
$$
\|Ef\|_{L^p(d\mu_R)} \lessapprox \|f\|_2, \quad \forall f\in L^2(B^n(0,1)), \quad \forall n\text{-dimensional measure } \mu.
$$

To facilitate the main inductive argument, we introduce the broad norm that replaces $L^p(d\mu_R)$. To pass from the generally weaker broad norm estimate to the $L^p$ estimate, see Subsection \ref{sec:BLpvsLp} below.

Let $K\gg 1$ be a large constant. We decompose $B^n(0,1)$ into balls $\tau$ of radius $K^{-1}$ and $B^{n+1}(0,R)$ into balls $B_{K^2}$ of radius $K^2$. Consider the decomposition of $f$:
$$
f = \sum_\tau f_\tau, \quad \text{where } f_\tau= f \chi_\tau.
$$
Let $A\ll K$ be a large constant. For a ball $B \subseteq B^{n+1}(0,R)$, we define
\begin{equation}
    \|Ef\|_{BL_A^p(B;d\mu_R)}^p := \sum_{B_{K^2}\subset B} \mu_{Ef}(B_{K^2}),
\end{equation}
where
$$
\mu_{Ef}(B_{K^2}):=\min_{V_1,...,V_A: 1-\text{subspaces of }\R^{n+1}}\left ( \max_{\tau: \text{Angle}(G(\tau),V_a)>K^{-1},  \forall a} \int_{B_{K^2}}|Ef_\tau|^p d\mu_R\right).
$$

We often use the shorthanded version $\|Ef\|_{BL_A^p(d\mu_R)}$ to mean $\|Ef\|_{BL_A^p(B^{n+1}(0,R);d\mu_R)}$ for simplicity. Since we are using $1$-subspaces $V_1,..,V_A$ in the definition of $BL^p$, it is also known as $2$-broad norm.

The only purpose of $A$ in the definition of the broad norm is to make certain versions of triangle inequality and H\"older's inequality possible. Even so, $BL^p_A(d\mu_R)$ is still not a \textit{norm}. The exact value of $A$ is not important, and hence we often write $BL^p$ instead of $BL^p_A$.

The broad norm was first introduced by Guth in \cite{Gu18} to derive an unweighted restriction estimate. Following \cite{DGOWWZ21}, $BL^p$ introduced here is a weighted variant of the $2$-broad norm in \cite{Gu18}. We will see in Lemma \ref{lem:BLptoLp} and Corollary \ref{cor:BWRE} that the $2$-broad estimate, despite being a weaker version of the corresponding bilinear estimate Theorem \ref{thm:BWRE}, acts equally well when being converted into the desired linear estimate.

\subsection{$BL^p$ estimates and $L^p$ estimates}\label{sec:BLpvsLp}

Following \cite{DGOWWZ21}, the lemma below allows us to pass $BL^p$ estimates to $L^p$ estimates.

\begin{lem}[c.f. \cite{DGOWWZ21}*{Lemma 5.2}] \label{lem:BLptoLp}
Let $n\geq 2$ and $p \geq \frac{2(n+1)}{n}$. Let $\mu$ be an $n$-dimensional measure in $\mathbb R^{n+1}$. Assume that for all $\varepsilon>0$, there exists large constant $A = A(\varepsilon)$ such that
\begin{equation}\label{eq:lem:BLptoLp:assum}
    \|Ef\|_{BL^p_A(d\mu_R)} \lesssim_{K,\varepsilon} R^{\varepsilon} \|f\|_2
\end{equation}
holds for all $K,R >1$.

Then, for all $\varepsilon>0$, $R > 1$, there holds
\begin{equation}
    \|Ef\|_{L^p(d\mu_R)}\lessapprox \|f\|_{2}.
\end{equation}
    
\end{lem}

Since we can control the broad part using \eqref{eq:lem:BLptoLp:assum}, we may assume that $\|Ef\|_{L^p(d\mu_R)}$ is not dominated by $\|Ef\|_{BL^p(d\mu_R)}$. In this case, locally $f$ must be concentrated in caps $\tau$ in the $K^{-1}$ neighborhood of $A$ many points, each of which is determined by a $1$-subspace $V_a$. There can only be $O_A(1)$ many such $\tau$, we then apply parabolic rescaling and then induction on $R$ to control the contribution from each $\tau$. For the details of the proof of Lemma \ref{lem:BLptoLp}, see \cite{DGOWWZ21}*{Lemma 5.2}.

\section{Weighted $L^p$ Restriction Estimates} \label{sec:WRE}
In this section, we introduce the linear and bilinear weighted restriction estimates established in  \cite{DGOWWZ21} and \cite{DLWZ}. 

\begin{thm}[Linear weighted restriction estimate, \cite{DLWZ} ]\label{thm:LWRE}
    Let $n\geq 2$, $m \in [3,n+1]$ and 
    $$
    p_{w,m}= \left (\frac{1}{2} -\frac{1}{m(m-1)(m-2)}\right )^{-1},  \quad \gamma_{w,m} = \frac{m-1}{2m} - \frac{n+m-1}{2m(m-1)(m-2)}.
    $$ Suppose that $Z = Z(P_1,.,,,.P_{n+1-(m-1)})$ is a transverse complete intersection, where $P_1,...,$ $P_{n+1-(m-1)}$ are polynomials of degrees bounded above by $D_Z$. Suppose further that $f \in L^2(B^n(0,1))$ is concentrated in wave packets from $\T_Z(N)$. Then for any $R\geq 1$ and $n$-dimensional measure $\mu$ in $\R^{n+1}$,
    \begin{equation}
        \|Ef\|_{L^{p_{w,m}}(d\mu_{R})} \lessapprox N^{O(1)} R^{\gamma_{w,m}} \|f\|_{2}.
    \end{equation}
\end{thm}

It is well-known that the broad norm $BL^p$ estimate is weaker than the corresponding $L^p$ estimate. As a corollary of Theorem \ref{thm:LWRE}, we prove the following broad norm version. 

\begin{cor}\label{cor:LWRE}
    Under the same assumptions as in Theorem \ref{thm:LWRE}, for any $A \geq 1$, we have
    \begin{equation}
        \|Ef\|_{BL_A^{p_{w,m}}(d\mu_{R})} \lessapprox N^{O(1)} R^{\gamma_{w,m}} \|f\|_{2}.
    \end{equation}
\end{cor}

\begin{proof}
    Let $\tau$ be balls of radius $K^{-1}$ and $f = \sum_{\tau} f_\tau$ as in the definition of broad norm in Subsection \ref{sec:weightedBL}. We have
    \begin{align*}
        \|Ef\|_{BL_A^{p_{w,m}}(d\mu_{R})} 
        &\leq \left (\sum_{\tau} \|Ef_\tau\|_{L^{p_{w,m}}(d\mu_R)}^{p_{w,m}} \right)^{1/p_{w,m}}\\
        &\lessapprox N^{O(1)} R^{\gamma_{w,m}} \left ( \sum_\tau \|f_\tau\|_{2}^{p_{w,m}} \right)^{1/p_{w,m}} \\
        &\leq N^{O(1)} R^{\gamma_{w,m}} \|f\|_{2}.
    \end{align*}
    The second-to-last inequality comes from applying Theorem \ref{thm:LWRE} on each $f_\tau$ while the last inequality comes from the $L^2$ orthogonality among $f_\tau$ and domination of $\ell^{p_{w,m}}$ by $\ell^2$ with $p_{w,m} \geq 2$.
\end{proof}

\begin{thm}[Bilinear weighted restriction estimate,  \cite{DGOWWZ21}*{Corollary 3.4}]\label{thm:BWRE}
    Let $n\geq 1$, $m \in [2,n+1]$ and $p_{b,w} = \frac{2(m+1)}{m}$. Suppose that $Z = Z(P_1,.,,,.P_{n+1-m})$ is a transverse complete intersection, where $P_1,...,P_{n+1-m}$ are polynomials of degrees bounded above by $D_Z$. Suppose further that $f_1,f_2 \in L^2(B^n(0,1))$ with supports separated by $\sim 1$ are concentrated in wave packets from $\T_Z(N)$. Then for any $R\geq 1$ and $n$-dimensional measure $\mu$ in $\R^{n+1}$,
    \begin{equation}
        \||Ef_1|^{1/2}|Ef_2|^{1/2}\|_{L^{p_{b,w}}(d\mu_R)} \lessapprox N^{O(1)} R^{-\frac{n+2-2m}{4(m+1)}} \|f_1\|_{2}^{1/2}\|f_2\|_{2}^{1/2}.
    \end{equation}
\end{thm}

It is also well-known that the $2$-broad norm $BL^p$ estimate is weaker than the corresponding bilinear $L^p$ estimate. See for instance \cite{GHI19}*{Proposition 6.4}.

\begin{cor}\label{cor:BWRE}
    Under the same assumptions as in Theorem \ref{thm:BWRE}, for any $A\geq 1$ and $f \in L^2(B^{n}(0,1))$ concentrated in wave packets from $\T_Z(N)$, we have
    \begin{equation}
        \|Ef\|_{BL_A^{p_{b,w}}(d\mu_{R})} \lessapprox N^{O(1)} K^{O(1)} R^{-\frac{n+2-2m}{4(m+1)}} \|f\|_{2},
    \end{equation}
    where $K$ is the large constant used to define the broad norm.
\end{cor}

The proof is included for completeness.
 
\begin{proof}
    The following argument works as long as $p \geq 2$. We denote $p_{b,w} = p$.
    Since 
    $$
    \|Ef\|_{BL_A^p(d\mu_{R})} \leq \|Ef\|_{BL_1^p(d\mu_{R})},
    $$
    it suffices to prove the case where $A=1$.
    
    Let $B_{K^2} \subset B^{n+1}(0,R)$ be a ball of radius $K^2$. Let $V$ and $\tau_1$ be the $1$-subspace and the $K^{-1}$ ball that realise the minimum and the maximum in the definition of $\mu_{Ef}(B_{K^2})$ respectively. We have
    $$
    \mu_{Ef} (B_{K^2}) = \|Ef_{\tau_1}\|_{L^p(B_{K^2},d\mu_R)}^p.
    $$
    
    We claim that there exists a $K^{-1}$ ball $\tau_2$ such that Angle$(G(\tau_1),G(\tau_2))\gtrsim K^{-1}$ and
    $$
    \mu_{Ef} (B_{K^2}) < 2\|Ef_{\tau_2}\|_{L^p(B_{K^2},d\mu_R)}^p.
    $$

    Suppose otherwise, the $1$-subspace that points to $G(\tau_1)$ would have realised a smaller $\mu_{Ef}(B_{K^2})$. We have shown the claim.

    In conclusion, we have found $\tau_1,\tau_2$ separated by $\gtrsim K^{-1}$ such that
    \begin{equation}\label{eq:temp3.5'}
        \mu_{Ef}(B_{K^2}) \lesssim \|Ef_{\tau_1}\|^{p/2}_{L^p(B_{K^2},d\mu_R)}\|Ef_{\tau_2}\|^{p/2}_{L^p(B_{K^2},d\mu_R)}.
    \end{equation}

    By pigeonholing, there are unit balls $B_1,B_2$ in $B_{K^2}$ centered on lattice points $(\frac{1}{2}\Z)^{n+1}$ such that
    $$
    \|Ef_{\tau_1}\|^{1/2}_{L^p(B_{K^2},d\mu_R)}\|Ef_{\tau_2}\|^{1/2}_{L^p(B_{K^2},d\mu_R)} \lesssim K^{O(1)}\|Ef_{\tau_1}\|^{1/2}_{L^p(B_1,d\mu_R)}\|Ef_{\tau_2}\|^{1/2}_{L^p(B_2,d\mu_R)}.
    $$
    
    On the other hand, $|Ef_{\tau_1}|$ and $|Ef_{\tau_2}|$ are essentially constants on $B_1$ and $B_2$ respectively by the locally constant property. Let $v=(v',v_{n+1}) = c(B_2)-c(B_1) \in \R^{n+1}$ be the difference of the centers of $B_1$ and $B_2$ so that $Ef_{\tau_2,v}$ on $B_1$ is a translated copy of $Ef_{\tau_2}$ on $B_2$, where 
    $$
    f_{\tau_2,v}(\xi): = f_{\tau_2}(\xi) e^{i(v' \cdot \xi + v_{n+1} |\xi|^2)}, \quad Ef_{\tau_2,v}(x ) = Ef_{\tau_2}(x+v).
    $$ 
    Therefore, $|Ef_{\tau_1}|$ and $|Ef_{\tau_2,v}|$ are essentially constants on $B_1$ and
    \begin{align*}
        \|Ef_{\tau_1}\|^{1/2}_{L^p(B_1,d\mu_R)}\|Ef_{\tau_2}\|^{1/2}_{L^p(B_2,d\mu_R)} &\sim \||Ef_{\tau_1}|^{1/2}|Ef_{\tau_2,v}|^{1/2}\|_{L^p(B_1,d\mu_R)} \\ &\leq \||Ef_{\tau_1}|^{1/2}|Ef_{\tau_2,v}|^{1/2}\|_{L^p(B_{K^2},d\mu_R)}. 
    \end{align*}

    In summary, we interchanged the geometric average and the $L^p$ norm over $B_{K^2}$ on the right hand side of \eqref{eq:temp3.5'} at a cost of $K^{O(1)}$ and a harmless modulation $f_{\tau_2,v}$:
    \begin{equation}
        \mu_{Ef}(B_{K^2}) \lesssim K^{O(1)}\||Ef_{\tau_1}|^{1/2}|Ef_{\tau_2,v}|^{1/2}\|_{L^p(B_{K^2},d\mu_R)}^p.
    \end{equation}

    Note that there are $K^{O(1)}$ choices for $\tau_1,\tau_2$ and $v$. Summing over $B_{K^2}, \tau_1,\tau_2$ and $v$, we obtain
    \begin{equation}\label{eq:temp3.7}
        \|Ef\|_{BL_1^{p}(d\mu_R)}^p\lesssim K^{O(1)} \sup_{\substack{v \in B(0,K^2) \\ K^{-1}\text{-separated caps }\tau_1,\tau_2 }}\||Ef_{\tau_1}|^{1/2}|Ef_{\tau_2,v}|^{1/2}\|_{L^p(d\mu_R)}^p . 
    \end{equation}
    Since $f_{\tau_1}$ and $f_{\tau_2,v}$ have supports $\tau_1$ and $\tau_2$ separated by $\gtrsim K^{-1}$, Theorem \ref{thm:BWRE} together with a parabolic rescaling gives
    \begin{equation}\label{eq:temp3.8'}
        \||Ef_{\tau_1}|^{1/2}|Ef_{\tau_2,v}|^{1/2}\|_{L^p(d\mu_R)} \lessapprox N^{O(1)} K^{O(1)} R^{-\frac{n+2-2m}{4(m+1)}} \|f_{\tau_1}\|_{2}^{1/2}\|f_{\tau_2,v}\|_{2}^{1/2}.
    \end{equation}
    By combining \eqref{eq:temp3.7}, \eqref{eq:temp3.8'} and the equality $\|f_{\tau_2,v}\|_2 = \|f_{\tau_2}\|_2$, we get
    \begin{align*}
        \|Ef\|_{BL_1^{p}(d\mu_R)}&\lessapprox N^{O(1)} K^{O(1)} R^{-\frac{n+2-2m}{4(m+1)}}  \sup_{K^{-1}\text{-separated caps }\tau_1,\tau_2} \|f_{\tau_1}\|_{2}^{1/2}\|f_{\tau_2}\|_{2}^{1/2} \\
        &\leq  N^{O(1)} K^{O(1)} R^{-\frac{n+2-2m}{4(m+1)}} \|f\|_2,
    \end{align*}
    as desired.
\end{proof}

\section{Polynomial Partitioning} \label{sec:PolyPart}

In this section, we prove the main inductive proposition via polynomial partitioning. 

Recall the definition of $p_m, \gamma_m$, $3\leq m \leq n+1$ in equations \eqref{eq:def_p_m,alpha_m} to \eqref{eq:p_n+1}. In addition, we define $p_2 = p_3$ and $\gamma_2 = \gamma_3$. As outlined in equations \eqref{eq:def_p_m,alpha_m} to \eqref{eq:p_n+1}, we have the following quantities:

\begin{equation}
    \alpha_m = 
    \begin{cases}
    1 \quad \quad &\text{ if } m=2; \\
    \frac{m^2-3m}{m^2-2m-2} \quad &\text{ if } 3\leq m \leq M; \\
    \alpha_{M+1}  = \frac{\gamma_{w,M+1}}{\gamma_{w,M+1}-\gamma_M} \quad &\text{ if } m=M+1; \\
    1 \quad &\text{ if } M+2 \leq m \leq n+1.
    \end{cases}
\end{equation}

\begin{equation}
    p_m = 
    \begin{cases}
    3 \quad &\text{ if } m=2,3; \\
    \frac{2m}{m-1} \quad &\text{ if } 4\leq m \leq M; \\
     \left (\frac{\alpha_{M+1}}{p_{M}} +  \frac{1-\alpha_{M+1}}{p_{w,M+1}}\right)^{-1} \quad &\text{ if } M+1 \leq m \leq n+1.
    \end{cases}
\end{equation}

\begin{equation}
    \gamma_m = 
\begin{cases}
    \frac{2-n}{12} \quad &\text{ if } m=2,3; \\
    \alpha_m \gamma_{m-1} + (1-\alpha_m) \gamma_{w,m} \quad &\text{ if } 4\leq m \leq M; \\
    0 \quad &\text{ if } M+1 \leq m \leq n+1.
\end{cases}
\end{equation}
Here $M$ is the largest integer in $[3,n+1]$ such that $\gamma_m$ is negative for all integers $m \in [3,M]$, and recall from Theorem \ref{thm:LWRE} that  $\gamma_{w,m}$ is the exponent of $R$ which appears on the right-hand side of the weighted $L^{p_{w,m}}$ estimate for $m-1$ dimensional algebraic variety in $\R^{n+1}$.

The exponents are chosen to meet the following three key properties:
\begin{equation} \label{eq:expo_prop1}
        p_m \geq \frac{2m}{m-1} \quad \text{for }m\geq 3,
    \end{equation}
\begin{equation}
        \gamma_m \leq 0,
    \end{equation} 
\begin{equation} \label{eq:expo_prop3}
        \frac{1}{p_m} = \frac{\alpha_m}{p_{m-1}} +  \frac{1-\alpha_m}{p_{w,m}}, \quad \gamma_m = \alpha_m \gamma_{m-1} + (1-\alpha_m) \gamma_{w,m}, \quad \alpha_m \in [0,1].
    \end{equation}
The last property can be interpreted as the pair $(\frac{1}{p_m},\gamma_m)$ lying on the line segment joining $(\frac{1}{p_{m-1}}, \gamma_{m-1})$ and $(\frac{1}{p_{w,m}}, \gamma_{w,m})$.

We first show that $p_{M+1} \geq \frac{2(M+1)}{M}$. Let $(\frac{M}{2(M+1)},\Tilde{\gamma})$ be the point lying on the line segment joining the points $(\frac{1}{p_M}, \gamma_M)$ and $(\frac{1}{p_{w,M+1}} , \gamma_{w,M+1})$. Since we define $M$ to be the largest integer so that $\gamma_M$ is negative, we have $\Tilde{\gamma} \geq 0$. Moreover, $(\frac{1}{p_{M+1}},\gamma_{M+1})$ is also on the same line segment with $\gamma_{M} < 0 = \gamma_{M+1} \leq \Tilde{\gamma} < \gamma_{w,M+1}$. We conclude that $\frac{M-1}{2M}<\frac{1}{p_{M+1}} \leq \frac{M}{2(M+1)} < \frac{1}{p_{w,M+1}}$. In particular, we have $p_{M+1} \geq \frac{2(M+1)}{M}$.

All other properties can be seen directly from the definition of $\alpha_m,p_m$ and $\gamma_m$.

\begin{prop}\label{prop:induction}
        Let $n\geq 3$. For all $\varepsilon>0$, there exists a large constant $\bar A >1$ and small constants $0<\delta_{\deg{}} \ll \delta\ll \delta_{n} \ll... \ll \delta_2 \ll \varepsilon$ such that the following holds. Let $m \in [2,n+1]$ be an integer, and $p_m,\gamma_m$ be defined above. Suppose that $Z = Z(P_1,.,,,.P_{n+1-m})$ is a transverse complete intersection, where $P_1,...,P_{n+1-m}$ are polynomials of degrees bounded above by $D_Z \leq R^{\delta_{\deg{}}}$. Suppose further that $f \in L^2(B^n(0,1))$ is concentrated in wave packets from $\T_Z(R^{\delta_m})$. Then for any $1\leq A \leq \bar A$, $R\geq 1$, $p\geq p_m$ and $n$-dimensional measure $\mu$ in $\R^{n+1}$,
        \begin{equation}\label{eq:prop:induction}
            \|E f \|_{BL^{p}_A(d\mu_R)} \leq C(K,\varepsilon,m,D_Z) R^{m\varepsilon} R^{\delta (\log \bar A - \log A)} R^{\gamma_m} \|f\|_2.
        \end{equation}
\end{prop}

In the statement of the proposition, the set of wave packets $\T_Z(R^{\delta_m})$ and the condition that $f$ is concentrated in wave packets from $\T_Z(R^{\delta_m})$ are defined in Subsection \ref{sec:tang_wavepackets}.


By Lemma \ref{lem:BLptoLp} and the relation $p_{n+1} \geq \frac{2(n+1)}{n}$, the $m = n+1$ case of Proposition \ref{prop:induction} implies Theorem \ref{thm:main_ext}. The remaining subsections are devoted to proving Proposition \ref{prop:induction}.

\subsection{The setup}

When $R$ is small, \eqref{eq:prop:induction} holds trivially by choosing a large enough implicit constant $C$. Moreover, the case $A = 1$ can be seen by a choosing large enough constant $\bar{A}$, depending on $\delta$. When $m=2$, Proposition \ref{prop:induction} is exactly Corollary \ref{cor:BWRE}. By induction, we may assume that \eqref{eq:prop:induction} holds with a smaller $A$, $R$ or $m$.

Following \cite{Gu18}, the case where most of the mass is concentrated near an $m-1$ dimensional variety is called \textit{algebraic}. More precisely, in the algebraic case, these is a transverse complete intersection $Y \subset Z$ of dimension $m-1$ defined by polynomials of degree at most $D(\varepsilon,D_Z)$, to be determined later, such that
\begin{equation}\label{eq:cellvsalg}
    \mu_{Ef} (N_{R^{1/2+\delta_m}}(Y) \cap B^{n+1}(0,
R)) \gtrsim \mu_{Ef} ( B^{n+1}(0,
R)).
\end{equation}
Otherwise, we call the case \textit{cellular}, or \textit{non-algebraic}.

The cellular case is considered in Subsection \ref{sec:cell_case} while the algebraic case is considered in Subsection \ref{sec:alg}.

\subsection{Cellular case}\label{sec:cell_case}

In this case, we assume that most of the mass is not concentrated near any $m-1$ dimensional variety, i.e. \eqref{eq:cellvsalg} does not hold for any transverse complete intersection $Y$ of one dimensional lower than $Z$.

Following \cite{Gu18}*{Section 8.1}, we can find a polynomial $P$ on $\R^{n+1}$ of degree at most $D=D(\varepsilon,D_Z)$ with the following properties:
$\R^{n+1} \setminus Z(P)$ is a union of $\sim D^m$ many open sets $O_i$, and 

\begin{equation}\label{eq:temp3.3}
    \|Ef\|_{BL_A^p(d\mu_R)}^p \lesssim D^m \|\chi_{O_i'}Ef\|_{BL_A^p(d\mu_R)}^p \lesssim D^m \|Ef_i\|_{BL_A^p(d\mu_R)}^p \quad \text{ for a definite proportion of }i, 
\end{equation}
where
\begin{equation*}
    O_i' = O_i \setminus  N_{R^{1/2+\delta}}(Z(P)) \quad \text{and } f_i = \sum_{(\theta,\nu):T_{\theta,\nu} \cap O_i' \neq \emptyset} f_{\theta,\nu}.
\end{equation*}

The key geometric observation is that each $T_{\theta, \nu}$ can pass through at most $D+1$ many $O_i'$. Therefore, by orthogonality we get
$$
\sum_{i} \|f_i\|_{L^2}^2 \lesssim D \|f\|_{L^2}^2.
$$
And by pigeonholing, there exists some $i'$ such that
\begin{equation}\label{eq:temp3.4}
    \|f_{i'}\|_{2}^2 \lesssim D^{1-m} \|f\|_{2}^2.
\end{equation}

Moreover, the induction hypothesis applied to $f_{i'}$ on balls of radius $R/2$ covering $B^{n+1}(0,R)$ gives
\begin{equation}\label{eq:temp3.5}
    \|Ef_{i'}\|_{BL_A^p(d\mu_R)} \leq 10^m C(K,\varepsilon,m,D_Z) R^{m\varepsilon} R^{\delta (\log \bar A - \log A)} R^{\gamma_m} \|f_{i'}\|_2
\end{equation}
for $p \geq p_m$.

Combining \eqref{eq:temp3.3}, \eqref{eq:temp3.4} and \eqref{eq:temp3.5}, we obtain
$$
\|Ef\|_{BL_A^p(d\mu_R)}^p \lesssim D^{m} D^{(1-m)\frac{p}{2}} \left (C(K,\varepsilon,m,D_Z) R^{m\varepsilon} R^{\delta (\log \bar A - \log A)} R^{\gamma_m} \|f\|_2\right)^p.
$$
The induction closes as long as $C D^{m + \frac{(1-m)p}{2}}\leq 1$, or $p > \frac{2m}{m-1}$ when $D=D(\varepsilon,D_Z)$ is chosen large enough. The endpoint case where $p= \frac{2m}{m-1}$ can be seen by H\"older's inequality and choosing $p$ larger than and close to $\frac{2m}{m-1}$. This finishes the proof for the cellular case as long as $p_m \geq \frac{2m}{m-1}$.

\subsection{Algebraic case}\label{sec:alg} In this case, we assume that most of the mass is concentrated near an $m-1$ dimensional transverse complete intersection $Y\subset Z$ defined by polynomials of degree at most $D(\varepsilon,D_Z)$, i.e.
\begin{equation}\label{eq:cellvsalg2}
    \mu_{Ef} (N_{R^{1/2+\delta_m}}(Y) \cap B^{n+1}(0,
R)) \gtrsim \mu_{Ef} ( B^{n+1}(0,
R)).
\end{equation}

In terms of wave packets, it suffices to consider $f$ concentrated in wave packets intersecting $N_{R^{1/2+\delta_m}}(Y)$. However, some wave packets may intersect $Y$ transversely. This situation is not covered by the induction hypothesis for $(m-1)$-dimensional algebraic varieties, which only gives estimates for $f$ concentrated in wave packets tangential to $Y$. The transverse sub-case is proved by induction on $R$, which requires re-doing the wave packet decomposition in a slightly smaller scale $\rho$. 

Define $\rho$ such that
$$
\rho^{\frac{1}{2}+\delta_{m-1}} = R^{\frac{1}{2} + \delta_m}.
$$
We subdivide $B^{n+1}(0,R)$ into a finitely overlapping collection of balls $B_j$ of radius $\rho$. We are interested in $B_j$ having non-empty intersection with the $R^{1/2+\delta_m}$-neighborhood of $Y$. Let $f_j$ be part of $f$ concentrated in wave packets intersecting such a $B_j$, i.e. $f_j = \sum_{(\tn) \in \T_j} f_{\tn}$, where
$$
\T_j := \{(\tn) : T_{\tn} \cap B_j \cap N_{R^{1/2+\delta_m}} (Y) \neq \emptyset\}.
$$

We further sort these wave packets according to whether they are tangent to $Y$ in $B_j$ or not. We say that $T_{\theta,\nu} \in \mathbb T_j$ is \emph{tangent} to $Y$ in $B_j$ if
\begin{equation}
T_{\theta,\nu}\cap 2B_j\subset N_{R^{1/2+\delta_m}}(Y) \cap 2B_j =
N_{\rho^{1/2+\delta_{m-1}}}(Y) \cap 2B_j
\end{equation}
and for any non-singular point $y\in Y \cap 2B_j\cap N_{10R^{1/2+\delta_m}}T_{\theta,\nu}$,
\begin{equation}
 \text{Angle}(G(\theta),T_y Y)\leq \rho^{-1/2+\delta_{m-1}}\,.
\end{equation}
Denote the tangent and transverse wave packets by
\[
\mathbb T_{j,{\rm tang}}:=\{(\theta,\nu)\in\mathbb T_j:\, T_{\theta,\nu} \text{ is tangent to } Y \text{ in } B_j\},\quad \mathbb T_{j,{\rm trans}}:=\mathbb T_j\setminus \mathbb T_{j,{\rm tang}},
\]and let
$$
f_\jtg = \sum_{(\tn) \in \T_\jtg} f_{\tn}, \quad f_\jtr = \sum_{(\tn) \in \T_\jtr} f_{\tn}.
$$

By \eqref{eq:cellvsalg2} and properties of $BL_A^p$ norm, we have
$$
\|Ef\|_{BL_A^p(d\mu_R)}^p \lesssim \sum_j \left ( \|Ef_\jtr \|_{BL_{A/2}^p(B_j;d\mu_R)}^p + \|Ef_\jtg \|_{BL_{A/2}^p(B_j;d\mu_R)}^p \right) + \rd(R)\|f\|_{2}^p.
$$

In the following two subsections, the transverse and the tangential terms will be estimated using induction on scale $R$ and dimension $m$.

\subsubsection{Transverse sub-case}\label{sec:alg_case1} Our goal in this subsection is to estimate the terms
$$
\sum_j \|Ef_\jtr \|_{BL_{A/2}^p(B_j;d\mu_R)}^p
$$
using induction on $R$.

First, we redo the wave packet decomposition for $f_\jtr$ at the smaller scale $\rho$. To apply induction hypothesis for $\rho < R$ directly to $f_\jtr$, we need $f_\jtr$ to be concentrated in scale $\rho$ wave packets $T_{\Tilde{\theta},\Tilde{\nu}}$ which are $\rho^{-1/2+\delta_m}$-tangent to $Z$ in $B_j$. However, this is not the case.

Note that $f_\jtr$ is concentrated in scale $R$ wave packets from $\T_Z(R^{\delta_m})$. Therefore, the angle $G(\Tilde{\theta)}$ the smaller scale wave packets $T_{\Tilde{\theta},\Tilde{\nu}}$ makes with $Z$ is at most $O(R^{-1/2+\delta_m} + \rho^{-1/2}) \leq \rho^{-1/2+\delta_m}$. We see that the angle condition is satisfied.

However, the farthest wave packets $T_{\Tilde{\theta},\Tilde{\nu}}$ are at a distance $\sim R^{1/2+\delta_m}$ from $Z$. Nevertheless, we can find a translation $Z_b = Z+b$ of $Z$ such that $T_{\Tilde{\theta},\Tilde{\nu}}$ is contained in $N_{\rho^{1/2+\delta_m}}(Z+b)$. Together with the angle condition, we conclude that $T_{\Tilde{\theta},\Tilde{\nu}}$ is $\rho^{-1/2+\delta_m}$-tangent to $Z_b$ in $B_j$. Above observations lead to the definition of $f_{\jtr,b}$:
$$
\T_{\jtr.b} : =\{(\Tilde{\theta},\Tilde{\nu}) : T_{\Tilde{\theta},\Tilde{\nu}} \cap B_j \cap N_{\rho^{1/2+\delta_m}}(Z_b) \neq \emptyset \}, \quad f_{\jtr,b} = \sum_{(\Tilde{\theta},\Tilde{\nu}) \in \T_{\jtr,b}} f_{\Tilde{\theta},\Tilde{\nu}}.
$$
By carefully choosing a set of translations $\{b\}$, we have
\begin{equation}\label{eq:temp3.8}
    \|Ef_\jtr \|_{BL_{A/2}^p(B_j;d\mu_R)}^p \lesssim (\log R) \sum_{b} \|Ef_{\jtr,b} \|_{BL_{A/2}^p(B_j;d\mu_R)}^p. 
\end{equation}
For details, see \cite{Gu18}*{Section 8.4}.

By the induction hypothesis at scale $\rho$, we get
\begin{equation}\label{eq:temp3.9}
    \|Ef_{\jtr,b} \|_{BL_{A/2}^p(B_j;d\mu_R)} \lesssim \rho^{m\varepsilon}\rho^{\delta(\log \Bar{A} - \log (A/2))} \rho^{\gamma_m} \|f_{\jtr,b}\|_{2}.
\end{equation}
Moreover, by \cite{Gu18}*{Lemmas 5.7, 7.4 and 7.5}, we have the following almost orthogonality
\begin{equation}
    \sum_{j,b} \|f_{\jtr,b}\|_2^2 \lesssim \sum_j\|f_\jtr\|_2^2 \lesssim \|f\|_2^2,
\end{equation}
and the equidistribution estimate
\begin{equation}
    \max_b \|f_{\jtr,b}\|_2^2 \lesssim R^{O(\delta_m)} \left(\frac{R^{1/2}}{\rho^{1/2}} \right)^{-(n+1-m)}\|f_\jtr\|_2^2.
\end{equation}
Therefore, we get
\begin{align}
    \sum_{j,b} \|f_{\jtr,b}\|_2^{p} 
    &\lesssim R^{O(\delta_m)} \left(\frac{R^{1/2}}{\rho^{1/2}} \right)^{-(n+1-m)(p/2-1)}\sum_{j}\|f_\jtr\|_2^{p} \notag \\
    &\lesssim R^{O(\delta_m)} \left(\frac{R^{1/2}}{\rho^{1/2}} \right)^{-(n+1-m)(p/2-1)}\|f\|_2^{p}.    \label{eq:temp3.12}
\end{align}

Combining \eqref{eq:temp3.8}, \eqref{eq:temp3.9} and \eqref{eq:temp3.12}, we get
$$
\sum_j \|Ef_\jtr \|_{BL_{A/2}^p(B_j;d\mu_R)}^p \lesssim R^{O(\delta_m)} \left (\rho^{m\varepsilon}\rho^{\delta(\log \Bar{A} - \log (A/2))} \rho^{\gamma_m}\right)^{p} \left(\frac{R^{1/2}}{\rho^{1/2}} \right)^{-(n+1-m)(p/2-1)}\|f\|_2^{p}.
$$
Recall that $R^{1/2+\delta_m} = \rho^{1/2+\delta_{m-1}}$. We have $\rho/R = R^{-O(\delta_{m-1})}.$ The terms with small parameter powers can be estimated by choosing $\delta \ll \delta_m \ll \varepsilon \delta_{m-1}$:
$$C R^{O(\delta_m)}\rho^{m\varepsilon}\rho^{\delta(\log \Bar{A} - \log (A/2))} \leq R^{m\varepsilon} R^{\delta(\log \bar{A} - \log(A))}. $$
To close the induction in the transverse sub-case, we want
$$
\rho^{p\gamma_m } \left(\frac{R^{1/2}}{\rho^{1/2}} \right)^{-(n+1-m)(p/2-1)} \leq R^{p\gamma_m } \quad \text{ for } p \geq p_m,
$$
which is equivalent to
\begin{equation}
    (n+1-m) (1/2-1/p_m) + 2 \gamma_m \geq 0.
\end{equation}

\subsubsection{Tangential sub-case}\label{sec:alg_case2} Our goal in this subsection is to estimate the terms
$$
\sum_j \|Ef_\jtg \|_{BL_{A/2}^p(B_j;d\mu_R)}^p 
$$
using induction on $R$ and $m$.

By the definition of $f_{\jtg}$, it is easy to check that $f_{\jtg}$ is concentrated in scale $\rho$ wave packets $T_{\Tilde{\theta},\Tilde{\nu}}$ which are $\rho^{-1/2+\delta_{m-1}}$-tangent to $Y$ in $B_j$, where $Y$ is an $m-1$ dimensional algebraic variety. The induction hypothesis at scale $\rho$ and dimension $m-1$ gives
\begin{equation}
    \|Ef_\jtg \|_{BL_{A/2}^{p_{m-1}}(B_j;d\mu_R)} \lesssim \rho^{(m-1)\varepsilon} \rho^{\delta(\log \Bar{A} - \log (A/2))} \rho^{\gamma_{m-1}} \|f_\jtg\|_2.
\end{equation}
On the other hand, we also have the weighted broad norm estimate from Corollary \ref{cor:LWRE}:
\begin{equation}
    \|Ef_\jtg \|_{BL_{A/2}^{p_{w,m}}(B_j;d\mu_R)} \lesssim \rho^{\varepsilon + O(\delta_{m-1})} \rho^{\gamma_{w,m}} \|f_\jtg\|_2.
\end{equation}

Let the pair $(\frac{1}{p_m},\gamma_m)$ lie on the line segment joining $(\frac{1}{p_{w,m}}, \gamma_{w,m})$ and $(\frac{1}{p_{m-1}}, \gamma_{m-1})$. For $p \geq p_m$, we get
\begin{equation}\label{eq:temp3.16}
    \|Ef_\jtg \|_{BL_{A/2}^{p}(B_j;d\mu_R)} \lesssim \rho^{(m-1)\varepsilon + O(\delta_{m-1})} \rho^{\delta(\log \Bar{A} - \log (A/2))} \rho^{\gamma_{m}} \|f_\jtg\|_2
\end{equation}
by logarithmic convexity, see \cite{Gu18}*{Lemma 4.2}. Here, we use a crude estimate for powers of $\rho$ with small parameters:
$$
\left (\rho^{(m-1)\varepsilon} \rho^{\delta(\log \Bar{A} - \log (A/2))} \right)^\alpha \cdot \left (\rho^{\varepsilon + O(\delta_{m-1})} \right)^{1-\alpha} \lesssim \rho^{(m-1)\varepsilon + O(\delta_{m-1})} \rho^{\delta(\log \Bar{A} - \log (A/2))}.
$$
for any $\alpha \in [0,1]$.

Taking the sum over all $B_j$ in $\ell^{p}$ norm, we obtain
\begin{align*}
    \left ( \sum_j \|Ef_\jtg \|_{BL_{A/2}^{p}(B_j;d\mu_R)}^{p} \right )^{1/p} 
    &\lesssim  \rho^{(m-1)\varepsilon + O(\delta_{m-1})} \rho^{\delta(\log \Bar{A} - \log (A/2))} \rho^{\gamma_{m}}\left ( \sum_j \|f_\jtg\|_2^{p} \right )^{1/p}  \\
    &\lesssim R^{(m-1)\varepsilon + \delta \log 2 + O(\delta_{m-1})} R^{\delta(\log \bar{A} - \log A)} R^{\gamma_m} \|f\|_2,
\end{align*}
where the last inequality follows from the fact that $\rho/R  = R^{-O(\delta_{m-1})}$ and the domination of the $\ell^{p}$ norm by the $\ell^2$ norm for $p > 2$. And by choosing $\delta \ll \delta_{m-1} \ll \varepsilon$, we have
$$
CR^{(m-1)\varepsilon + \delta \log 2 + O(\delta_{m-1})} \leq R^{m\varepsilon}.
$$

We have shown that the induction in the tangential sub-case is closed as long as $p\geq p_m$ and
\begin{equation}
    \text{the pair } (\frac{1}{p_m},\gamma_m) \text{ lies on the line segment joining } (\frac{1}{p_{w,m}}, \gamma_{w,m}) \text{ and } (\frac{1}{p_{m-1}}, \gamma_{m-1}).
\end{equation}

\subsection{Summary}
We summarize the conditions that $p_m$ and $\gamma_m$, $m\geq 3$, need to meet to close the induction:
\begin{itemize}
    \item Cellular case:
    \begin{equation}\label{eq:cond:cell}
        p_m \geq \frac{2m}{m-1};
    \end{equation}
    \item Algebraic case, transverse sub-case:
    \begin{equation}\label{eq:cond:trans}
        (n+1-m) (1/2-1/p_m) + 2 \gamma_m \geq 0;
    \end{equation}
    \item Algebraic case, tangential sub-case:
    \begin{equation}\label{eq:cond:tang}
    \text{the pair } (\frac{1}{p_m},\gamma_m) \text{ lies on the line segment joining } (\frac{1}{p_{w,m}}, \gamma_{w,m}) \text{ and } (\frac{1}{p_{m-1}}, \gamma_{m-1}).
    \end{equation}
\end{itemize}

\subsubsection{Admissiblility}
We first show that our choices of $p_m$ and $\gamma_m$ satisfy \eqref{eq:cond:cell} to \eqref{eq:cond:tang}. The conditions \eqref{eq:cond:cell} and \eqref{eq:cond:tang} follow immediately from the properties \eqref{eq:expo_prop1} and \eqref{eq:expo_prop3} respectively. It suffices to check the condition for the transverse sub-case.

Since $p_m \geq 2$, \eqref{eq:cond:trans} holds whenever $\gamma_m \geq 0$. Therefore, it suffices to consider the case $3\leq m \leq M$.

The case where $m=3$ can be verified directly:
$$
(n+1-m) (1/2-1/p_m) + 2 \gamma_m = (n+1 - 3) (1/2-1/3) + 2(2-n)/12 =0.
$$
By induction, we may assume that \eqref{eq:cond:trans} holds for $m-1$, which can be rearranged to
\begin{equation}\label{eq:temp3.27}
    (n+1-m) (1/2-1/p_{m-1}) + 2 \gamma_{m-1}  \geq  -(1/2-1/p_{m-1}) = -\frac{1}{2(m-1)}.
\end{equation}
On the other hand, from the definition of $p_{w,m}$ and $\gamma_{w,m}$, we have
\begin{equation}\label{eq:temp3.28}
    (n+1-m) (1/2-1/p_{w,m}) + 2 \gamma_{w,m}  = \frac{m-3}{m-2}.
\end{equation}
For $4\leq m \leq M$, we have
$$
    (\frac{1}{p_m},\gamma_m) =  \alpha(\frac{1}{p_{m-1}}, \gamma_{m-1}) + (1-\alpha) (\frac{1}{p_{w,m}}, \gamma_{w,m}),
$$
where 
$$\alpha = \frac{m^2-3m}{m^2-2m-2}.$$
Therefore, we can sum the $\alpha$ multiple of \eqref{eq:temp3.27} with the $(1-\alpha)$ multiple of \eqref{eq:temp3.28} to obtain:
\begin{align*}
    (n+1-m) (1/2-1/p_{m}) + 2 \gamma_{m}  &\geq -\frac{\alpha}{2(m-1)} + \frac{(1-\alpha)(m-3)}{m-2} \\
    &= \frac{(m-2)(m-3)}{2 m^3-6 m^2+4} \\ 
    &>0,
\end{align*}
where the last inequality follows from the following simple estimate on the denominator
$$
2m^3-6m^2 + 4  = 2m^2(m-3) + 4 > 0
$$
for $m\geq 4$. 

We have proved \eqref{eq:cond:trans}, and hence Proposition \ref{prop:induction}.

\subsubsection{Optimality} \label{sec:optimal} We now explain why these exponents give the smallest possible $p_{n+1}$ in Theorem \ref{thm:main_ext} using the methods and ingredients outlined in this paper. Our choice can be summarized as follows: we choose $p_m = \frac{2m}{m-1}$ saturating the constraint arising from the cellular case if the corresponding exponent $\gamma_m \leq 0$. 

In the figure below, we describe a typical scenario when picking the pair $(\frac{1}{p_m}, \gamma_m)$ lying on the line segment joining $(\frac{1}{p_{w,m}}, \gamma_{w,m})$ and $(\frac{1}{p_{m-1}}, \gamma_{m-1})$. The constraint arising from the cellular case \eqref{eq:cond:cell} requires $\frac{1}{p} \leq \frac{m-1}{2m}$ while that from the transverse case \eqref{eq:cond:trans} gives a lower bound on $\frac{1}{p}$ depending on $\gamma$. The pairs lying on the bold line between the points marked by \textit{cell} and \textit{trans} are admissible, satisfying all conditions \eqref{eq:cond:cell}-\eqref{eq:cond:tang}.

\begin{figure}[H]
    \centering
    \begin{tikzpicture}[scale = 0.8]
        \draw[->,thick] (-5,0)--(5.5,0) node[right]{$\gamma$};
        \draw[->,thick] (-5,0)--(-5,5.5) node[above]{$\frac{1}{p}$};
        \draw (0,0.2) -- (0,-0.2) node[anchor=north] {$0$};
        \draw (5,0.2) -- (5,-0.2) node[anchor=north] {$\frac{1}{2}$};
        \draw (-4.8, 5) -- (-5.2 , 5) node[anchor = east] {$\frac{1}{2}$};
        \draw (-4.8, 2.2) -- (-5.2 , 2.2) node[anchor = east] {$\frac{m-1}{2m}$};
        \draw[dashed] (-5,5 ) -- (5,5);
        \draw[dashed] (0,0 ) -- (0,5);

        \filldraw (-2,1.2) circle[radius=0.05] node[anchor=north] {\tiny $(\frac{1}{p_{m-1}} , \gamma_{m-1})$};
        \filldraw (2,4.2) circle[radius=0.05] node[anchor=south] {\tiny $(\frac{1}{p_{w,m}} , \gamma_{w,m})$} ;
        \filldraw (4,4.3) circle[radius=0.05] node[anchor=west] {\tiny $(\frac{1}{p_{w,m+1}} , \gamma_{w,m+1})$};
        \filldraw (-2/3,2.2) circle[radius=0.05] node[anchor=south east] {\tiny cell} ;
        \filldraw (-4/3,1.7) circle[radius=0.05] node[anchor= east] {\tiny trans};

        \draw[thin] (-2,1.2) -- (2,4.2);
        \draw[very thick] (-2/3,2.2) --  (-4/3,1.7); 
        \draw[dashed]  (-2/3,2.2) -- (4,4.3);
        \draw[dotted] (-4/3,1.7) -- (4,4.3);
        \draw[dashed] (-2/3,2.2) -- (-5,2.2);   
    \end{tikzpicture}
    \caption{An exaggerated diagram showing the relative positions of some important pairs of exponents}
\end{figure}
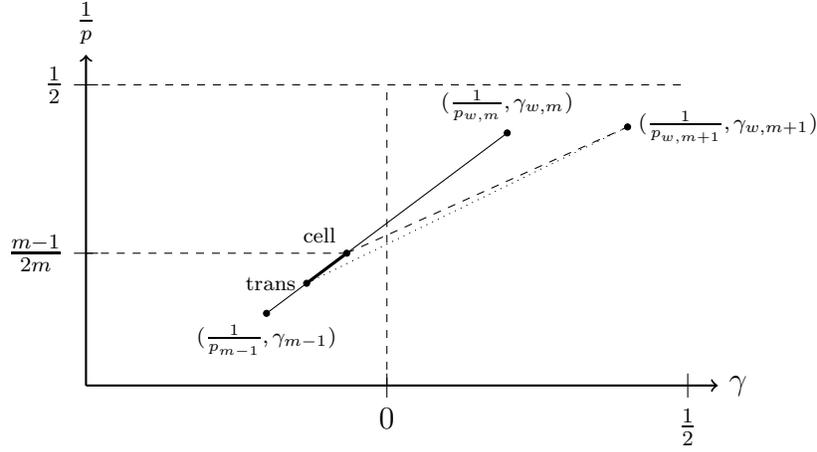

Note that the dashed line joining the point named \textit{cell} and $(\frac{1}{p_{w,m+1}} , \gamma_{w,m+1})$ is above the dotted line joining the point named \textit{trans} and $(\frac{1}{p_{w,m+1}}, \gamma_{w,m+1})$. This means that all estimates coming from the interpolation at the $m+1$ step will be better if we pick a larger $\frac{1}{p_m}$. Therefore, we should pick the exponents saturating the cellular constraint unless the corresponding $\gamma_m$ is positive, in which case we pick $p_m$ so that $\gamma_m=0$.

\bibliography{myreference}

\vspace{0.25cm}
	
\noindent Xiumin Du, Northwestern University, \textit{xdu@northwestern.edu}\\

\noindent Jianhui Li, Northwestern University, \textit{jianhui.li@northwestern.edu}

\end{document}